\documentclass[final]{siamltex}

\title{Genus of  curves in  smooth  hypersurfaces}

\author{ B. Wang}

\begin{document}

\maketitle

\begin{abstract}
This is the continuation of our paper [9].  In this paper which is self contained, we would like to give 
a different obstruction formula to the FIRST order deformation of the pair of a smooth curve and a smooth hypersurface. 
This obstruction formula leads to a genus formula for a smooth curve in a smooth hypersurface.
As an application 
we show that smooth elliptic curves  in  a smooth  hypersurface of degree 
$$h\geq 2n-1$$ in the projective space $\mathbf P^n, n\geq 3,$
can't deform in the first order to all hypersurfaces of the same degree.   In  particular, there are no smooth elliptic
curves in generic hypersurfaces of
degree $$h\geq 2n-1.$$ 
This application in return leads to a study of the deformation of the pair mentioned above.

\end{abstract}

\pagestyle{myheadings}
\thispagestyle{plain}
\markboth{ B. Wang}{Genus of curves in smooth hypersurfaces}

\section{Introduction}\par

Our study of the genus of a curve is originated from our study of the obstructions to  deformation of pairs
of varieties. We hope the numerical bounds and invariants we obtained can support the general study of the deformation of
pairs of varieties. 

 \medskip

Let's briefly introduce the deformation in consideration.  
Let $f_0\subset \mathbf P^n$ be a smooth hypersurface of degree $h$, where
$\mathbf P^n$ is the projective space of dimension $n\geq 3 $ over the complex numbers.
We'll denote the section in $H^0(\mathcal O_{\mathbf P^n}(h))$ that defines $f_0$ also
by $f_0$. Let $C_0\subset f_0$ be a smooth curve. We investigate the existence of a family of pairs
$C_t\subset  f_t$, the curves $C_t$ of degree $d$ and the hypersurfaces $f_t$ of degree $h$ in the projective space
$\mathbf P^n$\, where $t$ is in a variety. 
 We call $C_t\subset f_t$, a ``full" deformation of the pair. If $C_t, f_t$ are algebraic and $\{f_t\}_{all \ t} $ form an 
open set of the space of all hypersurfaces, then around $C_0, f_0$, they can be trimmed to a versal subvariety defined
 by Clemens and Ran in [3]. 
  A similar question was also investigated by L. Chiantini, Z. Ran 
in [4].  In general there is  a well-known Kodaira's deformation theory ([6])
about the fibred submanifold $C_0\subset f_0$ in a fibred complex manifold $f_0$, that says a sufficient condition for the 
$C_0$ to deform to all the other submanifolds  is $$H^1(N_{C_0}f_0)=0.$$ 
But in general, it is not clear that this condition is also a necessary condition, i.e if $H^1(N_{C_0}f_0)\neq 0$, 
$C_0$ may still be able to deform to all the other hypersurfaces (We don't have a proof of that yet).  In general, it may seem to be obvious
that $H^1(N_{C_0}f_0)=0$ is not a necessary condition for the pair to deform in all directions of the moduli space containing $f_0$, but
the situation could be very subtle if $f_0$ is a smooth hypersurface and $C_0$ is a curve, especially in the case where
$f_0$ has  a low dimension and $C_0$ has  a  low genus. 
This converse of the Kodaira's theorem reveals  subtle differences  in deformation theory of the pair of hypersurfaces and their
subvarieties. 
We are interested in the geometric difference between the existence of the first order deformation of the pair $C_0\subset f_0$ 
and
the existence of the ``full" deformation of the pair.  This paper is just the first step in this attempt, in which
 we prove theorem 1.2  below.  It gives a necessary condition (i.e. an obstruction) for $C_0$ to deform to ``other hypersurfaces" in the FIRST order. This condition in  theorem  1.2  below is different from that in 
[9].  The main application of it is the proof of the following,\bigskip

\begin{theorem}  
 There are no smooth elliptic curves  in  generic hypersurfaces of degree 
$$h\geq 2n-1$$ in the projective space $\mathbf P^n$. 
\end{theorem}
\bigskip

{\bf Remark} This theorem improves Clemens''s bound in [2] by 1.  But  this is only a weaker version of proposition 5.1 below.  Because the result in the theorem is noticeable and requires
 no buildups in definitions, we state it as the first theorem.  The simple bound in the theorem  involves many complex issues which may indicate the geometric difference in the deformation 
theory of a pair of projective varieties. See the remark after proposition 5.1.

\bigskip

The first order deformation of a pair of varieties were rigorously  defined and studied by Roy Smith and Robert Varley in [7].  Their main
interest lies in a pair of a smooth variety and its divisor,  and a sufficient condition for the pair to deform in the first order. Even though it is only for a divisor, it is still  a very important view and valid in many extensions. But we are going to bypass it in this paper because
we concentrate on a different situation. \bigskip

\noindent {\bf Setting for theorem 1.2.}\par

We need to give a formal description on the first order deformation of the pair. See [7] for a more general hypercohomological  approach. 
Let $H^0(\mathcal O_{\mathbf P^n}(h))$ denote the vector space of homogeneous polynomials of degree $h$ in $n+1$ variables.
We use the same letter $f_0\in H^0(\mathcal O_{\mathbf P^n}(h))$ to denote the hypersurface $div(f_0)\subset \mathbf P^n$,
homogeneous polynomial $f_0$, and its projectivization in  $\mathbf P(H^0(\mathcal O_{\mathbf P^n}(h)))$. 
Let $S\subset \mathbf P(H^0(\mathcal O_{\mathbf P^n}(h)))$ be a subvariety containing $f_0$ which is
a smooth point of $S$. Also assume that $f_0$ is a smooth hypersurface.  Let 
\begin{eqnarray} && X_S\subset S\times \mathbf P^n,\\&&
X_S=\{(f, x): f\in S, f(x)=0\}.\end{eqnarray}
be the universal hypersurface.
\par
Let $C$ be a smooth projective curve of genus $g$, and
$$c_0: C\to   f_0\subset P^n$$
a smooth imbedding of $C$ to $f_0$. Then 
$$\bar c_0: C\to \{f_0\}\times f_0\subset X_S$$ is the induced imbedding.
The projection $$P_S: X_S\to S$$ induces a map on the sections of bundles over $C$, 
\begin{equation} P_S^s: H^0(\bar c_0^\ast(TX_S)) \to T_{f_0}S,\end{equation}
where $ T_{[f_0]}S\simeq H^0(T_{[f_0]}S\otimes \mathcal O_C)$ is the space of global sections of the 
 trivial bundle whose each fibre is  $T_{f_0}S$. A  pre-image of $P_S^s$ represents a first order deformation of the pair.
\medskip

In this paper we consider two specific parameter spaces for $S$:

\smallskip

{\bf Assumption (1) } The first subvariety $S$ under consideration  is the collection of hypersurfaces 
in the following form:
$$ f_0+\sum_{i=0}^h a_i L_0\cdots\hat L_i\cdots L_{h},\quad (\hat L_i \ is \ omitted)$$
where $L_i\in H^0(\mathcal O_{\mathbf P^n}(1)), i=0, \cdots, h$ are fixed sections whose zeros  are distinct, i.e.
\begin{equation} div(L_i)\neq  div(L_j), i\neq j. \end{equation}
Let $$A'=\mathbf C^{h+1}=\{(a_0, \cdots, a_h)\}$$ be the parameter space of the family. Let
$A\subset A'$ that parametrizes smooth hypersurfaces. So $S=A$ in this case.\medskip

{\bf Assumption (2) }  Secondly $S$ is the entire space $\mathbf P(H^0(\mathcal O_{\mathbf P^n}(h)))$. We will denote
$\mathbf P(H^0(\mathcal O_{\mathbf P^n}(h)))$ by $E$.  So $S=E$ in this case.
\bigskip

The exact formula for the genus of the curve in the hypersurface will depend on the hypersurface. It is not surprised 
to see the genus is a semi-continuous function on the space of hypersurfaces. Thus our genus formula will involve the space $A$. 
Let's first introduce the term reflecting this dependence. Continuing with the notations for the assumption (1), let

\begin{equation} u_i=L_0{\partial\over \partial a_0}-L_i{\partial\over \partial a_i}, i=1, \cdots, h
\end{equation} be sections
of $ TA\otimes \mathcal O_{\mathbf P^n}(1)$. It is easy to see $u_i$ annihilate the universal polynomial $F$,
$$F(a, x)=f_0(x)+\sum_{i=0}^h a_i L_0(x)\cdots\hat L_i(x)\cdots L_{h}(x),\quad (\hat L_i \ is \ omitted).$$ 
Hence $u_i$ are tangent to
$X_A$ at all points of $X_A$. Let $G(1)$ be the sub-sheave generated by $u_i$. 
We have an imbedding map of sheaves, 
$$\begin{array}{ccc} \bar c_0^\ast(G(1)) & \rightarrow &  \bar c_0^\ast(TX_A(1)).\end{array}$$
Let $\phi_3$ be the induced map on their $H^1$ groups,
$$\begin{array}{ccc} H^1(\bar c_0^\ast(G(1))) & \stackrel{\phi_3}\rightarrow &  H^1(\bar c_0^\ast(TX_A(1)))\end{array}$$

We'll use the notations: 
  $h^i(E)$ denotes the dimension of $H^i(E)$ for any sheaf $E$. For any linear map $\alpha$, $Im(\alpha)$ denotes
the image
of the map $\alpha$. Let $N_{c_0}V$ denote the pull-back of any subbundle $V$ of $T\mathbf P^n|_{C_0}$ to $C$. 
Let $\mathcal L=c_0^\ast(\mathcal O_{\mathbf P^n}(1))$.  Note $d=deg(\mathcal L)$.
\bigskip

\begin{theorem} Let $f_0, C_0, A$ be as above.  Let \begin{equation} \{L_i=0\}\cap \{L_j=0\}\cap C_0=\emptyset, \quad i\neq j.\end{equation}
Also 
assume $P_A^s$ is surjective.    Then   
\begin{equation} \begin{array}{cc}  \sigma(c_0, f_0, A):=& (h-2n)d+(n-1)(g-1)\\
&+h^0(c_0^\ast(Tf_0(1))+(h+1)h^1(\mathcal L)\\
&-dim(Im(\phi_3))-h^1(c_0^\ast(T\mathbf P^n(1)) =0,\end{array}\end{equation}
\end{theorem}\bigskip

In this formula, the most difficult term is $h^0(c_0^\ast(Tf_0(1))$, but the most intriguing term is $dim(Im(\phi_3))$. 
To understand $dim(Im(\phi_3))$ better, we 
introduce two more terms $m, k$.
 Let  $m$ be the dimension of the image $B$ of the following composition map $\phi_6\circ \phi_0$,
\begin{equation}\begin{array}{ccccc} H^0(\bar c_0^\ast(G(1))) & \stackrel{\phi_0}\rightarrow &  H^0(\bar c_0^\ast(TX_A(1)))&
\stackrel{\phi_6}\rightarrow & H^0(N_{c_0}f_0(1)))
\end{array}\end{equation}
where $\phi_6$ is induced from  the composition map of bundles over $C$, 
$$\bar c_0^\ast(TX_A(1))\stackrel{\pi}\rightarrow c_0^\ast(Tf_0(1))\to N_{c_0}f_0(1),$$
where the projection map $\pi$ is an important map that is induced from the splitting in the formula (3.3) below, and the $\pi$ exists only because
$C_0$ can deform to all the other hypersurfaces in $A$ in the first order, i.e. $P_A^s$ is surjective.
 So $$B=Im(\phi_6\circ \phi_0).$$ 
Now consider the map
\begin{equation}\begin{array} {ccc}
H^1(\oplus_m\mathcal O_C) &\stackrel{\phi_5}\rightarrow &H^1(N_{c_0}f_0(1)))\end{array}\end{equation}
where $\phi_5$ is induced from the bundle map of the trivial bundle to the normal bundle $c_0^\ast(N_{c_0}f_0(1))$, 
$$\oplus_m\mathcal O_C\simeq B\otimes \mathcal O_C\to c_0^\ast(N_{c_0}f_0(1)) .$$
  
 Let $k=dim(ker(\phi_5))$.  Note that
$B$ is unique up to an isomorphism but $k, m$ are uniquely determined by the sections $L_0, \cdots, L_h$ and
the  first order deformations of $C_0$ to hypersurfaces collected in $A$.\bigskip

Applying theorem  1.2  to smooth  curves in a smooth hypersurfaces, we obtain that  \bigskip

\begin{corollary} (\em Genus formula) \par
Continuing from theorem 1.2 (with the same assumptions), 
let $g$ be the genus of a smooth curve $C_0$ in a general hypersurface of degree $h$ in
$\mathbf P^n$. In addition, we assume $d> 4(g-1)$. 
Then
$$g={(h-2n+1)d+h^0(N_{C_0}f_0(1))-n+4+k\over m-n+4}.$$

\end{corollary}

\bigskip

{\bf Remark}. \par
(1)   Theorem 1.2  proves that if $\sigma(c_0, f_0, A)\neq 0$,  
then $C_0$ can't deform to all the hypersurfaces in $A$ in 
the first order.  Thus
$\sigma(c_0, f_0, A)\neq 0$  gives us an obstruction to the 
deformations of $C_0$ to other hypersurfaces. However the name, ``obstruction" may be misleading because
 $\sigma(c_0, f_0, A)$ depends on $A$ and $\phi_3$, i.e. depends how $C_0$ deforms to other hypersurfaces in $A$.
\smallskip

(2) There are lots of work on the bound of genus of the subvariety of a generic hypersurface or a generic
complete intersection.
We don't mean to include a complete list of results in this area. We only mention those that have a direct relation 
with our results.  Corollary  1.3 has some overlap with the results of Clemens in [2], where he showed that
a lower bound of genus is 
$$ {1\over 2} (h-2n+1) d+1.$$
But his bound is not sharp (see [8]),  it implies 
that there are no immersed elliptic curves of in a generic hypersurface of degree
$$h\geq 2n.$$

\bigskip

In section 2 below, we describe and prove a theorem of H. Clemens' on the deformation of hypersurfaces. 
This is the starting point for the entire paper. We include it here in its completeness because it is not published and we need to use it in this paper. 
In section 3, we study the deformation of the curve $C_0$ with
the deformation of hypersurface to derive a necessary condition of the pair to deform in the first order. 
 In section 4, we apply the result from section 3 to obtain a genus formula for a smooth curve on a smooth
hypersurface (no need to be generic). In section 5, we apply the genus formula to obtain the bounds of hypersurfaces
that imply theorem 1.1. 

\bigskip

\section{Deformation of the hypersurface}  The main idea of the proof is to transform the problems
of $T\mathbf P^n$ to similar types of problems of some isomorphic bundle ${TX_A(1)\over G(1)}$. 
The exact sequence from this quotient,  in return,  gives a way to the study of $T\mathbf P^n$.
The existence of the
first order deformations of the pair $C_0, f_0$  allows us to derive properties in this
exact sequence. 
 Thus the isomorphism between 
$T\mathbf P^n$ and ${TX_A(1)\over G(1)}$ serves as an important bridge between two different realms. 
In this section, we introduce the construction of the vector bundle ${TX_A(1)\over G(1)}$, provided and proved by Herb Clemens ([1]). The curve $C_0$ is not involved. 

\medskip

Recall $L_0, \cdots, L_h\in H^0(\mathcal O_{\mathbf P^n}(1))$ satisfy the formulas (1.4) as in
assumption (1), and
 \begin{equation} F(a_1, \cdots, a_h, x)=f_0(x)+\sum_{i=0}^h a_i L_0(x)\cdots\hat L_i(x)\cdots L_{h}(x), \quad (omit \ L_i)\end{equation}
is the universal polynomial. Thus
$$\{F=0\}=X_A\subset A\times \mathbf P^n.$$
is also the universal hypersurface, which is smooth. 
Let $W\subset \mathbf P^n$ denote the complement of the proper  subvariety
\begin{equation} \cup_{h\geq  j>i\geq 0}\{L_i=L_j=0\}.\end{equation}
Let  \begin{eqnarray} && X_W=X_A\cap (A\times W)\\&&
f_0^W=f_0\cap W.\nonumber\end{eqnarray}
Recall
\begin{equation} u_i=L_0{\partial\over \partial a_0}-L_i{\partial\over \partial a_i}, i=1, \cdots, h
\end{equation} are sections
of $ TA\otimes \mathcal O_{W}(1)$. Since $u_i$ annihilate $F$, they are tangent to
$X_W$. So let 
\begin{equation} G(1)\subset  TX_W(1)\end{equation}
be the vector bundle of rank $h$ over $X_W$ that is generated by the sections
$u_i$.  Note that because of the condition (2.2) on $L_j, j=0, \cdots, h$, 
$G(1)$ is a trivial bundle of rank $h$ over $X_W$.
 
\bigskip
For any smooth varieties $V_1, V_2$, let
$$T_{V_1/V_2}$$ denote the relative tangent bundle of $V_1$ over $V_2$, i.e. it is
the bundle $TV_1\oplus \{0\}$ over the variety $V_1\times V_2$. 
\bigskip

The following theorem 2.1 is communicated to us by H. Clemens ([1]), who after learning our
construction of the section $u_i$,  proved: 
\bigskip

\begin{theorem} ({\rm H. Clemens})
\begin{equation} {TX_W(1)\over G(1)}\simeq T_{W/A}(1),
\end{equation}
where $T_{W/A}(1)$ is restricted to $X_W$. 

\end{theorem}
\begin{proof} Consider the exact sequence
\begin{equation}\begin{array}{ccccccccc}
0&\rightarrow {TX_W(1)\over G(1)}&\rightarrow {T(A\times W)(1)\over G(1)} 
&\rightarrow &\mathcal D &\rightarrow 0.\end{array}\end{equation}
of bundles over $X_W$, 
where $\mathcal D$ is some quotient bundle over $X_W$.
Easy to see \begin{equation} c_1(\mathcal D)=c_1(\mathcal O_{\mathbf P^n}(h+1))|_{X_w}.\end{equation}
Let $s$ be a generic section of $\mathcal O_{\mathbf P^n}(1)$ that does not have common zeros with
$L_i, i=0, \cdots, h$. Let $\sigma$ be the reduction of  $s{\partial \over \partial a_0}$
in ${T(A\times W)(1)\over G(1)}$. Notice the zeros of $\sigma$ is exactly
\begin{equation} div(\sigma)=div(s L_1\cdots L_h).\end{equation}
Since $s L_1\cdots L_h\in H^0(\mathcal O_{\mathbf P^n}(h+1))$, $\sigma$ splits the sequence (2.7).
If $L_s\subset {T(A\times W)(1)\over G(1)}$ is the line bundle generated by $\sigma$, 
\begin{equation} L_s\oplus {TX_W(1)\over G(1)}={T(A\times W)(1)\over G(1)},\end{equation}
as bundles over $X_W$.
Secondly, we have another exact sequence
\begin{equation}\begin{array}{ccccccccc}
0&\rightarrow T_{W/A}(1)&\rightarrow {T(A\times W)(1)\over G(1)} 
&\rightarrow &\mathcal D' &\rightarrow 0.\end{array}\end{equation}
of bundles over $X_W$,
where $\mathcal D'$ is some quotient bundle over $X_W$.
By the direct calculation (note $G(1)$ is a trivial bundle):
$$c_1(\mathcal D')=c_1(c_0^\ast(T_{A/W}(1)))=(h+1)(c_1(\mathcal O_{\mathbf P^n}(1)))|_{X_W}$$
As above,  $\sigma$ splits this sequence (2.11). Hence
\begin{equation} L_s\oplus T_{W/A}(1)={T(A\times W)(1)\over G(1)}.\end{equation}
Comparing the formulas  (2.10), (2.12), we obtain
\begin{equation} {TX_W(1)\over G(1)}\simeq T_{ W/A}(1), \end{equation}
over $X_W$.\quad 
\end{proof}

\bigskip

\section{Deformations of curves to other hypersurfaces}

In this section, we prove theorem 1.2 and a formula for $dim(Im(\phi_3))$.
\bigskip

\begin{proof} of theorem 1.2: 
In  theorem 1.2, sections $L_1, \cdots, L_h$ satisfy both conditions in formulas (1.4) and (1.6).
Consider the exact sequence of Clemens' quotient ${TX_W(1)\over G(1)}$, 
$$\begin{array}{ccccccccc}0 & \rightarrow &\bar c_0^\ast(G(1))& \rightarrow &
 \bar c_0^\ast(TX_W(1))& \rightarrow &\bar c_0^\ast({TX_W(1)\over G(1)})
& \rightarrow & 0\end{array} $$
This induces the long exact sequence \begin{equation}\begin{array}{ccccccc}
H^0(\bar c_0^\ast(TX_A(1))) & \stackrel{\phi_1}\rightarrow & H^0( \bar c_0^\ast({TX_W(1)\over G(1)})) &&&& \\
& &\downarrow\scriptstyle{\phi_2}&&&& \\
&& H^1(\bar c_0^\ast(G(1)))&&&& \\
& &\downarrow\scriptstyle{\phi_3} &&&&\\ &&
H^1(\bar c_0^\ast(TX_A(1)))
& \stackrel{\phi_4}\rightarrow & H^1( \bar c_0^\ast({TX_W(1)\over G(1)}))& \rightarrow & 0.\end{array}\end{equation}

This exact sequence and theorem (2.1) which says
$$c_0^\ast({TX_A(1)\over G(1)})\simeq c_0^\ast(T\mathbf P^n(1)),$$ 
yield  
\begin{equation} dim(Im(\phi_3))+h^1(c_0^\ast(T\mathbf P^n(1))-h^1(c_0^\ast(TX_A(1)))=0.\end{equation}
 
Next we calculate $h^1(c_0^\ast(TX_A(1)))$. Since  $P_A^s$ is surjective, we obtain
\begin{equation} c_0^\ast(TX_A(1))\simeq (\oplus_{h+1}\mathcal L) \oplus c_0^\ast(Tf_0(1)), \end{equation}
 where 
each copy $\mathcal L$ is $ \mathcal L\simeq \mathcal O_C\otimes \mathcal L$,  and the trivial bundle $\mathcal O_C$ 
is generated by the section $${\partial \over \partial a_k}-\beta_k, k=0, \cdots, h,$$
where $a_k$ are affine coordinates of $A'$ defined in the assumption (1). 
We choose one  $\beta_k\in c_0^\ast(T\mathbf P^n)$ for each ${\partial \over \partial a_k}$.
Then 
\begin{equation}
h^1(\bar c_0^\ast(TX_A(1)))=(h+1)h^1(\mathcal L)+h^1(c_0^\ast(Tf_0(1)))
\end{equation}
Using Riemann-Roch, we obtain that

\begin{eqnarray}   &&h^1(c_0^\ast(Tf_0(1))) \\
 &&=h^0(c_0^\ast(Tf_0(1)))-\biggl(Ch(c_0^\ast(f_0(1)))\cdot Tod(TC)\biggr)\nonumber\\
 &&=h^0(c_0^\ast(Tf_0(1)))-\biggl(c_1(c_0^\ast(Tf_0(1)))+{n-1\over 2}(TC)\biggr)\nonumber\\
 &&=h^0(c_0^\ast(Tf_0(1)))-\biggl(c_1(c_0^\ast(T\mathbf P^n(1)))-(h+1)d+{n-1\over 2}c_1(TC)\biggr)\nonumber\\
&&=h^0(c_0^\ast(Tf_0(1)))+(h-2n)d+(n-1)(g-1)\nonumber
\end{eqnarray}
\par

Combining  formulas  (3.2), (3.4) and (3.5), we proved theorem 1.2.
\end{proof} 

\bigskip

\begin{lemma} Assume $P_A^s$ is surjective and $d> 4(g-1)$.  Then
$$dim(Im(\phi_3))=mg-k,$$
where $k=dim(ker(\phi_5))$ (see formula (1.9)).
\end{lemma}

\begin{proof} Because $P_A^s$ is surjective, we have the decomposition  as in formula (3.3), 
$$c_0^\ast(TX_A(1))\simeq \oplus_{h+1} \mathcal L \oplus c_0^\ast(Tf_0(1)).$$
Then $$H^1(c_0^\ast(TX_A(1))\simeq \oplus_{h+1}H^1(\mathcal L)\oplus H^1(c_0^\ast(Tf_0(1))).$$
Since $d>2( g-1)$, $H^1(\mathcal L)=0$. Thus 
$$H^1(c_0^\ast(TX_A(1))\simeq  H^1(c_0^\ast(Tf_0(1))).$$
 Notice that
$$m=dim(B),$$
is the dimension of the image of $\phi_6\circ \phi_0$, i.e.,  
$$B=span (\{c_0^\ast(L_0\beta_0-L_k\beta_k)\}_{all\ k})\subset H^0(c_0^\ast(N_{c_0}f_0(1))),$$
where
$c_0^\ast(L_0\beta_0-L_k\beta_k)$ are reduced to $H^0(N_{c_0}f_0(1))$.
Now consider the commutative diagram
$$\begin{array} {ccccc}
H^1(\oplus_m\mathcal O_C) &\stackrel{\phi_7}\rightarrow &H^1(c_0^\ast(Tf_0(1)))\simeq H^1(c_0^\ast(TX_A(1))
&\stackrel{\phi_3} \leftarrow & H^1(c_0^\ast(G(1)))\\
\downarrow\scriptstyle{\phi_5} && \downarrow\scriptstyle{P} &&\\
H^1(c_0^\ast(N_{c_0}f_0(1))) &= &H^1(c_0^\ast(N_{c_0}f_0(1))) &&
\end{array}$$
where $\phi_7$ is induced from the bundle map of the trivial bundle 
$$B\otimes \mathcal O_C\simeq \oplus_m\mathcal O_C$$
over $C$ to $c_0^\ast(Tf_0(1))$.\footnote{$\phi_7$ is well-defined, because $B\otimes \mathcal O_C$ is a 
a trivial bundle.} 
 There is an exact sequence for the second vertical map $P$,
\begin{equation}\begin{array} {ccccc} H^1(TC(1)) &\rightarrow &
H^1(c_0^\ast(Tf_0(1))) &\stackrel{P}\rightarrow &H^1(c_0^\ast(N_{c_0}f_0(1))).
\end{array}\end{equation}
Because $d>4(g-1)$ in both cases where $g=0$ and $g \neq 0$,  $ H^1(TC(1))=0$. 
Thus $P$ is injective. Then we obtain  
\begin{equation} dim(Im(\phi_7))=mg-dim(ker(\phi_5))=mg-k.\end{equation}
 Then it suffices to prove that
$$Im(\phi_3)=Im(\phi_7).$$

Let $\{U_j\}$ be an affine open  covering of $C$. Let
$$\{\epsilon_k^{j_1 j_2}\}, k=1, \cdots, h$$ be the representative of an element in 
the \u{C}ech-cohomology for $$H^1(c_0^\ast(G(1)))\simeq H^1(\oplus_h \mathcal O_C).$$

By the definition of $\phi_3$, the image of $\phi_3$ is just the co-cycle
$$\{\sum_k \epsilon_k^{j_1j_2}(c_0^\ast(L_0\beta_0-L_k\beta_k))|_{U_{j_1}\cap U_{j_2}}\}\in H^1(c_0^\ast(Tf_0(1)),$$
where $c_0^\ast(L_0\beta_0-L_k\beta_k)$ is regarded as a section in $H^0(c_0^\ast(Tf_0(1))$ 
(without modular $TC_0$ as for $B$). Notice $H^1(TC(1))=0$. Then
we have the decomposition 
$$H^0(c_0^\ast(Tf_0(1)))\simeq H^0(TC(1))\oplus H^0(c_0^\ast(N_{c_0}f_0(1))).$$
This decomposition shows that any cycle in  $H^0(c_0^\ast(N_{c_0}f_0(1)))$ could have a representative 
in $H^0(c_0^\ast(Tf_0(1)))$.  Then it suffices to show
that the co-cycle $$\alpha=\{\sum_k \epsilon_k^{j_1j_2}(c_0^\ast(L_0\beta_0-L_k\beta_k))|_{U_{j_1}\cap U_{j_2}}\}$$
is zero if the global sections $L_0\beta_0-L_k\beta_k$ are tangent to $C_0$.  This is indeed true because $P$ is injective. More specifically, if 
all $L_0\beta_0-L_k\beta_k$ are tangent to $C_0$, 
$$\alpha\in Image(H^1(TC\otimes \mathcal L))\subset  H^1(c_0^\ast(Tf_0(1))).$$
Again because $d>4(g-1)$, $H^1(TC\otimes \mathcal L)=0$. Thus $\alpha=0$. 
 \end{proof}

\bigskip

\section{Genus formula for smooth curves in  smooth  hypersurfaces in $\mathbf P^n$}\par

In this section we apply  theorem 1.2  to study the genus of curves $C_0$  in a smooth hypersurface $f_0$  in $\mathbf P^n$.
\par
We are going to prove corollary  1.3: \par

\begin{proof} of  corollary 1.3 : 

Because $deg(\mathcal L^\ast\otimes K)=-d+2g-2<0$, by the Serre-duality,
 $$h^1(\mathcal L)=h^0(\mathcal L^\ast\otimes K)=0.$$

Consider the twisted Euler sequence pulled back to $C$:
$$\begin{array}{ccccccccc}
0 & \rightarrow &\mathcal O_C\otimes c_0^\ast(\mathcal O_{\mathbf P^n}(1)) &\rightarrow &
(\oplus_{n+1} c_0^\ast (\mathcal O_{\mathbf P^n}(2))) &\rightarrow &
c_0^\ast(T\mathbf P^n(1))
&\rightarrow & 0.\end{array}$$
Then we have the exact sequence on cohomologies
\begin{equation}\begin{array}{ccccc}
H^1(\oplus_{n+1} c_0^\ast (\mathcal O_{\mathbf P^n}(2)))&\rightarrow &
H^1(c_0^\ast(T\mathbf P^n(1)))&\rightarrow & H^2(\mathcal O_C\otimes \mathcal O_{\mathbf P^n}(1))\end{array}\end{equation}
Since $d> g-1$, $ H^1(\oplus_{n+1} c_0^\ast (\mathcal O_{\mathbf P^n}(2)))=0$. By the Grothendieck vanishing theorem ([5]), 
 $H^2(\mathcal O_C\otimes c_0^\ast(\mathcal O_{\mathbf P^n}(1)))=0$. Thus $ H^1(c_0^\ast(T\mathbf P^n(1)))=0$.
It follows from  formula (1.7) and  lemma 3.1 that
\begin{equation} (h-2n)d+(n-1)(g-1)+h^0(c_0^\ast(Tf_0(1))-mg+k=0.\end{equation}

Consider the exact sequence
$$\begin{array}{ccccccccc}
0 & \rightarrow &TC\otimes \mathcal L &\rightarrow &
c_0^\ast(Tf_0(1)) &\rightarrow &
c_0^\ast(N_{c_0}f_0(1))
&\rightarrow & 0.\end{array}$$
It induces 
$$\begin{array}{ccccccccc}
0 & \rightarrow & H^0(TC\otimes \mathcal L )&\rightarrow &
H^0(c_0^\ast(Tf_0(1))) &\rightarrow &
H^0(c_0^\ast(N_{c_0}f_0(1)))
&\rightarrow & 0.\end{array}$$
Hence
\begin{equation} h^0(c_0^\ast(Tf_0(1))=h^0(c_0^\ast(N_{c_0}f_0(1)))+d+3-3g.\end{equation}
Combining formulas (4.2), (4.3), we complete the proof.

\end{proof}

\section{Smooth elliptic curves in a smooth hypersurfaces in $\mathbf P^n$}\par

In this section, we turn our attention to elliptic curves. \medskip

\begin{proposition} Assume $f_0$ is a smooth hypersurface of degree $h$ in $\mathbf P^n$ and $C_0$ is a smooth elliptic curve in $f_0$.  
\par
(1) Let $A$ be the parameter space
of hypersurfaces containing $f_0$ as in theorem 1.2. If  $P_A^s$ is surjective, then
$h\leq 2n-1$.\par

(2) If $P_E^s$ is surjective, then $h\leq 2n-2$

\end{proposition}
\medskip

Theorem 1.1 follows from proposition 5.1, because if there are smooth elliptic curves in generic hypersurfaces
$f_0$, then $P_E^s$ is surjective. Then proposition 5.1, part (2) says  $h\leq 2n-2$. This is the same assertion as
that in theorem 1.1.
\bigskip

{\bf Remark} \par
(1) A Clemens' theorem in [2] implies that there are no smooth elliptic curves  in  generic hypersurfaces of degree 
$$h\geq 2n$$ in the projective space $\mathbf P^n$. We use our method, theorem 1.2
 to improve Clemens' inequality
by $1$. \par
(2) Furthermore, our bound $2n-2$ only requires the first order deformation of the pair $C_0\subset f_0$. 
This also means the bound obtained with the condition of ``full" deformation of the pair may be sharper than our bound. 
To obtain a better bound, one may have to use higher order deformations of the pair $C_0\subset f_0$. This is indeed the case in [8] for rational curves, and in [3] for elliptic curves in sextic 3-folds. \par
Thus our bound comes from the existence of the abstract first order deformation, while Clemens and Ran's better bound (or Voisin's for rational curves) for
$n=4$  comes from the existence of  the ``full" deformation of the pair.  The different bounds represent the different deformations of
the pair.

\par
(3) Our inequality is not under  the ``full" deformation condition. Thus it is mostly not sharp if the ``full" deformation
of the pair is assumed. 
There are examples showing this: in the case of $n=4$, Clemens and Ran proved that there are no elliptic curves in generic sextic 
three-folds ([3]) (with the  assumption of ``full" deformation).  But it was conjectured by J. Harris and proved by G. Xu that our bound is sharp  for $n=3$ ([10]) with the assumption  of the ``full" deformation.  Thus we speculate 
the sharp bound of $h$ is not a polynomial in $n$ under any deformation conditions. 
This is contrary to the case of rational curves (the sharp upper bound  with the ``Full" deformation is $2n-3$ ). Also we should point it out that, in the proof, if we only require the weaker bound $h\geq 2n$, the $f_0$ only needs to deform to hypersurfaces in $A$ in the first order. \bigskip

\begin{proof} proof of proposition 5.1:   Let $C_0$ be a smooth elliptic curve in a smooth hypersurface $f_0\subset \mathbf P^n$ in 
$A$. By the genus formula in corollary 1.3, 
$$g={(h-2n+1)d+h^0(N_{C_0}f_0(1))-n+4+k\over m-n+4}.$$
Thus $$(h-2n+1)d=m-h^0(N_{C_0}f_0(1))-k.$$
Since $m\leq h^0(N_{C_0}f_0(1))$, $h\leq 2n-1$.

This shows a generic hypersurface $f_0$ in the family $A$ can't have a smooth elliptic curve 
if $h\geq 2n$.\footnote{ If $f_0$ is generic, this also can be derived from Clemens' result in [2]. But our $f_0$ is
in $A$ and is not generic.}

To prove  part (2), we only need  to come up with a contradiction for 
$h=2n-1$.    From genus formula, we have $$g={h^0(N_{c_0}f_0(1))-n+4+k\over m-n+4}.$$
Because $k\geq 0$, it suffices to prove that
$$h^0(N_{c_0}f_0(1))>m=dim(B).$$
(This contradicts $g=1$).
Note $B\subset H^0(N_{c_0}f_0(1))$. We would like to construct a section in $H^0(N_{c_0}f_0(1))$ but not in $B$.

This uses  the entire space $\mathbf P_E$ of hypersurfaces. 
Notice $m$ is determined by the sections $L_i\in H^0(\mathcal O_{\mathbf P^n}(1))$.
Let's carefully choose these sections $L_i$. Since  $P^s_E$ is surjective, then for $\alpha\in  H^0(\mathcal O_{\mathbf P^n}(h))$, there is a section $<\alpha>\in H^0(c_0^\ast(T\mathbf P^n))$ such that
$$(\alpha, <\alpha>)\in H^0(\bar c_0^\ast(TX_E)).$$
It is clear that  $<\alpha>$ is unique upto a section of $c_0^\ast(Tf_0)$.  First fix a point $q=c_0(t_0)\in C_0$.  
 Let $E_q$ be any fixed 
hyperplane in $T_qf_0$. Let
$L_0$ be a section in $$H_q := H^0(\mathcal O_{\mathbf P^n}(1)\otimes \mathcal I_q)$$  where
$\mathcal I_x$ is the ideal sheaf of $\{q\}\subset \mathbf  P^n$.   We define
\begin{equation}\begin{array} {c} S_{E_q}\subset H_q\times \biggl( H^0(\mathcal O_{\mathbf P^n}(1))\biggr)^{h}\\
S_{E_q}=\{(L_0, L_1, \cdots, L_h): <L_0L_1\cdots \hat L_i\cdots L_h>|_q\in E_q, i\neq 0\}
\end{array}\end{equation}
Next we claim
\par
{\bf Claim 5.1}: {\em   there are $E_q$ and 
$$L_0\in H_q , L_1,\cdots, L_h\in H^0(\mathcal O_{\mathbf P^n}(1))$$
such that
$$ (L_0, L_1, \cdots, L_h)\in S_{E_q}$$
and $L_0, L_1, \cdots, L_h$ satisfy formula (1.6), i.e.
$$\{L_i=0\}\cap \{L_j=0\}\cap C_0=\emptyset, i\neq j.$$
}
Let's prove the claim. For the worst, we may assume $H^0(c_0^\ast(Tf_0))=0$.\footnote{If $H^0(c_0^\ast(Tf_0))\neq 0$, we
should fix a decomposition of the linear space
$$H^0(c_0^\ast(T\mathbf P^n))=H^0(c_0^\ast(Tf_0))\oplus V.$$
Then take $<\alpha>$ to be the $V$-component of the inverse image of $\alpha$ in the decomposition. Such
$<\alpha>$ is unique.}
Then
$<\alpha>_q$ is a well-defined vector in $T\mathbf P^n|_q$ for any point $q\in C_0$. First 
for generic $L_0\in H_q$, and   generic 
$J,  L\in H^0(\mathcal O_{\mathbf P^n}(1))$, 
$$<L_0  L\cdots L>_q, <L_0 J L\cdots L>_q $$ are linearly independent 
vectors. Because 
if it was not true, by the genericity of all sections, $ <L_0 J L\cdots L>_q$ would've been  zero. Since 
$$L^{h-2}\in H^0(\mathcal O_{\mathbf P^n}(h-2))$$ 
for all generic $L$ linearly span the entire space
$$ H^0(\mathcal O_{\mathbf P^n}(h-2)).$$
Thus by the linearity again,  we see that  for any $$\alpha\in  H^0(\mathcal O_{\mathbf P^n}(h)),$$
$<\alpha>$ would've been  zero at the points of $C_0$ where $<\alpha>$ lies in $Tf_0$.  This is not true
because by the $GL(n+1)$ action on $\mathbf P_E\times \mathbf P^n$, for any $(n+1)\times (n+1)$ matrix $g$ with the trace being zero, 
$(-gf_0, gq)$ lies in $TX|_q$ (this is the infinitesimal action). \par
Now we can choose a subspace $E_q\subset Tf_0|_q$ of dimension $n-2$ such that $ <L_0  L\cdots L>_q$ lies in $E_q$ but 
$<L_0 J L\cdots L>_q$ does not. Also choose such $L_0, L$ that
they do not vanish simultaneously at any point of $C_0$.  Next we would like to prove that
$S_{E_q}$ is smooth at $(L_0, L, \cdots,  L)$. To show this, we consider the 
analytic subset $U_i=\{L+x_i J\}$ in each  $H^0(\mathcal O_{\mathbf P^n}(1)$, where
$x_i$ are complex numbers. Let $L_0'\in H_q$ be another generic 
section, and $U_0=\{L_0+x_0 L_0'\}$ such that
$$<L_0'L\cdots L>_q\not\in E_q.$$
Then $$S_{E_q}\cap (U_0\times U_1\times U_2\cdots \times U_h),$$
is a subset of  $\mathbf C^{h+1}$ parametrized by $\{(x_0, \cdots, x_h)\}$,   that is defined by
\begin{equation} g(x_0, \cdots, \hat x_i, \cdots, x_h)=0, i=1, \cdots, h.\end{equation}
for some multi-linear polynomial $g$ in $h$ variables. 
Let $$a'={<L_0'L\cdots L>\over E_a}\in {Tf_q\over E_q}\simeq \mathbf C, $$ and 
$$a={<L_0J L\cdots L>\over E_a}\in {Tf_q\over E_q}\simeq \mathbf C.$$
By our choice, $a\neq 0$ and $a'\neq 0$. The Jacobian matrix of the formula (5.2) at the origin(corresponding to
$(L_0L \cdots L)$) is 
\begin{equation}\left (\begin{array}{cccccc}   a' & a &a &\cdots & a &0\\
a' &0  &a &\cdots & a &a\\
 \vdots &\vdots &\vdots &\ddots & \vdots &\vdots \\
a' &a &\cdots &a & 0 &a
\end{array}\right)\end{equation}
This matrix is row-reduced to
\begin{equation} \left (\begin{array}{cccccc}   a' & a &0 &\cdots & a &0\\
0 & -a  &0 &\cdots & 0 &a\\
 \vdots &\vdots &\vdots &\ddots & \vdots &\vdots \\
0&0 &\cdots &0 & -a &a
\end{array}\right )\end{equation}
which has the full rank. This shows $ S_{E_q}$ is smooth 
at $(L_0, L, \cdots,  L)$. Next we shrink $ S_{E_q}$ around $(L_0, L, \cdots, L)$ to make it irreducible. That is to let 
$$S_{E_q, \epsilon}=S_{E_q}\cap (U_0^\epsilon \times U^\epsilon\cdots U^\epsilon)$$
for sufficiently small $\epsilon\in \mathbf C$, where $U^\epsilon$ is  
the open disk of $ H^0(\mathcal O_{\mathbf P^n}(1)$  centered at  $(L_0, L\cdots,   L)$ with radius $\epsilon$ and
$$U_0^\epsilon=\{L_0+x_0 L_0': |x_0|\leq \epsilon\}.$$
It is clear that $S_{E_q, \epsilon}$ is symmetric under the permutations of $L_1, \cdots, L_h$.
For any $L_1, L_2\in U^\epsilon$ which have no common zeros along $C_0$, by the dimension count and similar 
infinitesimal argument as in the formula (5.3), we can find other sections $L_0\in  H_q$ and 
$L_3, \cdots, L_h$ such that
$(L_0, L_1, L_2, \cdots, L_h)\in S_{E_q, \epsilon}$, i.e. the projection of $S_{E_q, \epsilon}$ to the second and third components, 
$$U^\epsilon\times U^\epsilon$$ is surjective. 
Since $S_{E_q, \epsilon}$ is irreducible (because it is smooth) and symmetric, we proved that for the generic
point $(L_0, L_1,\cdots,  L_h)\in S_{E_q, \epsilon}$, $L_i, L_j, i\neq j, i\neq 1\neq j$ do not have
common zeros along $C_0$. Also $L_0$ does not have common zeros with any of other $ L_i$ along $C_0$ because
$L_0$ does not have common zeros with $L$ along $C_0$ for the center $(L_0, L, \cdots, L)\in S_{E_q, \epsilon}$.
This proves the claim 5.1.
Let  $ L_0,  L_1, \cdots,  L_h$ satisfy the claim 5.1.  Also let 
$$L_0< L_1\cdots  L_h>-L_k<L_0  L_1\cdots \hat L_k\cdots  L_h>
\in H^0(c_0^\ast(Tf_0(1))), $$
lie in $E_q$ at $q$ for all $k\neq 0$ (where $\hat \cdot$  means ``omitting"). Then we apply the 
sections $L_0,  L_1, \cdots, L_h$ to construct 
the subspace $A$ as in section 1. We obtain the integer $m$ which is the dimension of corresponding $B$. Then each
section $\beta\in B$ must lie in $E_q$ at $q$. Next we construct a section of
$H^0(c_0^\ast(N_{c_0}f_0(1))$ not in $B$. 
 By the $GL(n+1)$ action on $\mathbf P_E\times \mathbf P^n$, there are
sections $L_1', \cdots, L_{h-1}'$ such that 
$$<L_0 L_1'\cdots L_{h-1}'>\in H^0(c_0^\ast(T\mathbf P^n))$$ 
does not lie in $E_q$ at $q$, where $L_0\in H_q$.  Let $L_h'\in H^0(c_0^\ast(T\mathbf P^1)) $ be any section not in
$H_q$. 
This shows that 
$$L_0<L_1'\cdots L_{h}'>-L_h'<L_0 L_1'\cdots L_{h-1}'>$$ is reduced to a non-zero section in
$H^0(c_0^\ast(N_{c_0}f_0(1))$, but it is not in 
$$B\subset H^0(c_0^\ast(N_{c_0}f_0(1))),$$
because it does not lie in $E_q$ at $q$. 
Thus 
$$dim(H^0(c_0^\ast(N_{c_0}f_0(1))))>m=dim(B).$$
We complete the proof.\medskip

\end{proof}

\section*{Acknowledgments}
We would like to thank H. Clemens for his generous help and constant encouragement, especially for his enlightening 
communication of  theorem (2.1).


\begin{thebibliography}{10} 
 

\bibitem{C2}{\sc H. Clemens}, 
{\em Private letters}, 2010. 
 
\bibitem{C1} {\sc H. Clemens}, 
{\em Curves in generic hypersurfaces}, 
Ann. Sci. \'Ecole Norm. Sup. 19(1986), pp. 629-636



\bibitem{HCZR}{\sc H. Clemens and Z. Ran}, 
{\em Twisted genus bounds for
subvarieties of  generic hypersurfaces}, American Journal of Mathematics 
 126(2004), pp. 89--120.




\bibitem{LCZR} {\sc L. Chiantini and Z. Ran}, 
{\em  Subvarieties of generic hypersurfaces in any
variety}, Math. Proc.Camb.Phil. Soc. 130 (2001), pp. 259-268.



\bibitem{Grothendieck}{\sc A. Grothendieck}, 
{\em Sur quelques points d'alg\'ebre homologique
}, Tohoku Math. J. 9(1957), pp. 119--221


\bibitem{Ko}{\sc K. Kodaira}, 
{\em On stability of compact submanifolds of complex manifolds}, Amer. J. Math 85(1963), 
  pp. 79-94.

\bibitem{Roy}{\sc R. Smith and R. Varley},
{\em Deformation of theta divisor and the rank 4 quadrics problem}, Comp. Math. 76 (1990),
pp. 367--398.


\bibitem{V} 
{\sc  C. Voisin}, {\em On a conjecture of Clemens on rational curves on hypersurfaces}, J. of Differential Geometry  44
(1996), pp. 200--213.



\bibitem{Wa}{\sc B. Wang}, 
{\em Obstructions to  deformation of curves to other hypersurfaces}, Preprint, 2011

\bibitem{Xu}{\sc G. Xu}, 
{\em Subvarieties of general hypersurfaces in projective space}, J. of Differential Geometry  39(1994), 
  pp. 139-172.



\end{thebibliography}
\end{document}